\documentclass[11pt]{amsproc}

\usepackage{amsmath, amsthm,amsfonts,amscd,amssymb, setspace}
\usepackage{a4wide}
\usepackage{amsrefs}

\newtheorem{theorem}{Theorem}[section]
\newtheorem{lemma}[theorem]{Lemma}
\newtheorem{proposition}[theorem]{Proposition}

\theoremstyle{definition}

\theoremstyle{remark}

\theoremstyle{remark}

\theoremstyle{remark}

\theoremstyle{remark}

\newcommand{\Zb}{{\mathbb Z}}
\newcommand{\F}{{\mathbf F}}

\newcommand{\Nb}{{\mathbb N}}

\renewcommand{\phi}{\varphi}

\title{About the length of laws for finite groups}

\author{Andreas Thom}
\address{Andreas Thom, TU Dresden, Germany}
\email{andreas.thom@tu-dresden.de}

\begin{document}

\onehalfspace

\begin{abstract}
We prove new upper bounds of the form $O(n /\log(n)^{2-\varepsilon})$ for the length of laws that hold for all groups of size at most $n$ -- improving on previous results of Bou-Rabee and Kassabov-Matucci. The methods make use of the classification of finite simple groups. Stronger bounds are proved in case the groups are assumed to be nilpotent or solvable. 
\end{abstract}

\maketitle

%\tableofcontents

\section{Introduction}

We denote the free group on two generators $a,b$ by $\F_2=\langle a,b \rangle$. For every group $G$, a word $w \in \F_2$ determines a word map $w \colon G \times G \to G$ by evaluation. We say that $w$ is a law for $G$ if the image of the corresponding word map on $G$ consists only of the neutral element -- that is, $w(g,h)=1$ for all $g,h \in G$. Every finite group admits laws, for example $a^n \in \F_2$ is a law for $G$ for $n:=|G|$. Recently, there has been some interest in finding short laws for specific groups of families of groups. This was first done systematically by Hadad for finite simple groups of Lie type in \cite{MR2764921}. Bou-Rabee \cite{MR2574859} started a program to determine the so-called residual finiteness growth of various groups, which also involved the study of laws for symmetric groups ${\rm Sym}(n)$ (see \cite{brm} and independently \cite{MR2972333}) and the family of all groups (resp.\ all nilpotent or all solvable groups) of size at most $n$. The best upper bounds for length of laws of ${\rm Sym}(n)$ are currently due to Kozma and the author \cite{kozthom}, where the proof makes use of consequences of the Classification of Finite Simple Groups (CFSG).

The aim of this note is to provide a construction of a non-trivial word $w_n \in \F_2$ of length $O(n /\log(n)^{2-\varepsilon})$ which is a law for every finite group of size at most $n$. We will first do this separately for solvable groups and semi-simple groups (in the sense of Fitting) and then combine the two results to obtain a result for general finite groups. This answers a question from Kassabov-Matucci \cite[Question 18]{MR2784792}. Our main result is stated as Theorem \ref{main} at the end of the paper. Its proof is spread over the entire paper, combining results on nilpotent and solvable groups with the case of semi-simple groups and results of Fitting. In the course of the proof, we also provide new upper bounds for the the length of laws that hold for all nilpotent resp.\ all solvable groups of size at most $n$. 
 
Previously, bounds on the length of laws that hold for all groups of size at most $n$ have been obtained by Bou-Rabee \cite{MR2574859} who constructed such a law of length $O(n^3)$. Later, Kassabov and Matucci \cite{MR2784792} improved the bound to $O(n^{3/2})$ using an elementary but technical result of Lucchini on permutation groups from \cite{MR1722784}. Our methods rely on the CFSG and thus use a machinery which is considerably more heavy than the methods used before.

The only known lower bound comes from the observation that ${\rm PSL}_2(p)$ does not satisfy any law shorter than $p$, see \cite{MR2532876}. Since the size of ${\rm PSL}_2(p)$ is roughly $p^3$, this implies that no law of length $o(n^{1/3})$ can hold for all groups of size at most $n$. We conjecture that the bound $n/\log(n)^2$ proved in this paper  is sharp up to negligible factors -- see also the discussion after Theorem \ref{main}.

\section{Preliminaries}

Let $G$ be a group. For a word $w \in \F_2$ we set $Z(G,w):= \{(g,h) \in G \times G \mid w(g,h)=1 \}$ and call it the vanishing set of $w$. Clearly, $w$ is a law for $G$ if and only if $Z(G,w)=G \times G$.
We denote the minimal word length of a word $w \in \F_2$ when written in terms of the generators $\{a,a^{-1},b,b^{-1}\}$ by $|w|$.

The following lemma is the key in many steps of the argument. It allows to combine finite sets of words and increase the vanishing set. We called it the commutator lemma, and it has appeared in different forms in the work of Hadad \cite[Lemma 3.3]{MR2764921} and Kassabov-Matucci \cite[Lemma 10]{MR2784792}. In the following form, it was proved as \cite[Lemma 2.2]{kozthom}.

\begin{lemma} \label{commut}
Let $w_1,\dotsc,w_m$ be non-trivial words in $\mathbf{F}_2$. Then there exists a non-trivial word $w\in\mathbf{F}_k$ such that for all groups $G$ we have:
$$Z(G,W) \supset Z(G,w_1) \cup \dots \cup Z(G,w_m).$$
Moreover, the length of $w$ is
bounded by $16\cdot m^2\max|w_i|$.
\end{lemma}

Lemma \ref{commut} can be used to construct laws for a single group but also laws for families of groups. Indeed, applying it to a laws for $G$ and $H$ yields a law that holds for both $G$ and $H$. We will apply the lemma in both ways.

The preceding lemma is enough to prove the upper bound of Bou-Rabee on laws that hold for all finite groups of size at most $n$. Indeed, every element of a group of order at most $n$ has order at most $n$. Hence, $G \times G = \cup_{k=1}^n Z(G,a^k)$. The lemma provides a law of length $16n^3$. The result of Kassabov-Matucci starts with this observation and treats groups with an element of order exceeding $n^{1/2}$ separately in order to obtain the improved bound $O(n^{3/2})$, see \cite{MR2784792}.

\vspace{0.2cm}

The following lemma allows us to deal with extensions of groups.
\begin{lemma} \label{ext}
Let $1 \to H \to G \to G/H \to 1$ be an extension of groups. Let $w' \in \F_2$ be a non-trivial law for $H$ and $w'' \in \F_2$ be a non-trivial law for $G/H$. There exists a non-trivial law $w \in \F_2$ for $G$ of length bounded by $|w'| (|w''|+2)$.
\end{lemma}
\begin{proof}
We set $w(a,b):= w'(w''(a,b), xw''(a,b)x^{-1})$, where we choose $x \in \{a,a^{-1},b,b^{-1}\}$ so that $w''(a,b)$ and $x w''(a,b) x^{-1}$ are free in $\F_2$. This proves the claim.
\end{proof}

The preceding lemma will allow us to decompose a group into its solvable radical and the corresponding semi-simple quotient. But we will also apply it to study automorphism groups of finite simple groups and reduce the case of semi-simple groups to that of simple groups.

\section{Nilpotent and solvable groups}

Let $G$ be a group. We define the lower central series of $G$ by setting $\gamma_1(G)=G$ and defining $\gamma_{k+1}(G)=[G,\gamma_{k}(G)]$ for all $k \geq 1$. Similarly, the derived series is defined by $G^{(0)} := G$ and $G^{(k+1)} := [G^{(k)},G^{(k)}]$ for all $k \geq 0$. A group is called {\it nilpotent} if $\gamma_{k+1}(G)=1$ for some $k \in \Nb$ and the smallest such $k$ is called the nilpotency class of $G$. A group is called {\it solvable} if if $G^{(k)}=1$ for some $k \in \Nb$ and the smallest such $k$ is called the solvability class of $G$.
It is a classical result of Hall that $G^{(k)} \subset \gamma_{2^k}(G)$ for all $k \geq 0$.

It is well-known that the nilpotency class of a nilpotent group can be estimated by the size. More precisely, we have that
$\gamma_k(G)=1$ whenever $k> \log_2(n)$. Indeed, this is obvious since $|\gamma_m(G)/\gamma_{m+1}(G)| \geq 2$, whenever $\gamma_m(G) \neq 1$. By a result of Elkasapy and the author \cite[Theorem 2.2]{elkthom}, there exists a non-trivial word $v_k \in \F_2$ of length at most $O(k^{3/2})$ which lies in $\gamma_k(\F_2)$. This implies the following proposition, compare \cite[Theorem 5.1]{elkthom} for a slightly sharper bound.

\begin{proposition} \label{nil}
Let $n \in \Nb$. There exists a word $v_n \in \F_2$ of length bounded by $O(\log(n)^{3/2})$ which is a law for all nilpotent groups of size at most $n$.
\end{proposition}

In order to treat solvable groups, we need to recall a few more notions. Let $G$ be a solvable group of size at most $n$. By a result of Fitting, there exists a unique maximal, normal, and nilpotent subgroup of $G$. This group is called the Fitting subgroup and we denote it by $N \lhd G$. See \cite{hall} for general background on this and related concepts that will come up.
It is a classical fact that the canonical map $\alpha \colon G \to {\rm Aut}(N)$, given by the conjugation action, is injective.

Since $N$ is nilpotent, $N$ is a product of $p$-groups $N = \prod_p N(p)$, and we obtain also ${\rm Aut}(N)= \prod_p {\rm Aut}(N(p))$. For the $p$-group $N(p)$, the Frattini subgroup $\Phi(N(p))$ is equal to $N(p)^p[N(p),N(p)]$ and we denote the Frattini quotient by $V(p) = N(p)/\Phi(N(p))$. The quotient group $V(p)$ is a vector space over the finite field $\Zb/p\Zb$. Since $|V(p)| \leq |N(p)| \leq n$, the dimension $m(p):=\dim_{\Zb/p\Zb} V(p)$ of $V(p)$ over $\Zb/p\Zb$ is bounded by $\log_2(n)$.

We consider the natural homomorphism $\alpha_p \colon {\rm Aut}(N(p)) \to {\rm Aut}(V(p)) = {\rm GL}_{m(p)}(p)$ with $m(p) \leq \log_2(n)$. Clearly, the image of $G$ in ${\rm GL}_{m(p)}(p)$ is solvable since $G$ is assumed to be solvable. 
By a result of Zassenhaus \cite[Satz 7]{MR3069692}, the solvability class of a linear group is logarithmic in the dimension of the vector space on which it acts.
Later, Huppert \cite{MR0089851} proved that the solvability class of a solvable subgroup of ${\rm GL}_{m(p)}(p)$ cannot exceed $1 + 7\log_2 (m(p))$. Newman \cite[Theorem A$_S$]{MR0302781} determined the optimal bound to be $5\log_9(m(p)) + O(1)$.
Thus, we obtain that the solvability class of the image of $G$ in $\prod_p {\rm Aut}(V(p))$ cannot exceed $5 \log_9 (2) \cdot \log_2 (\log_2(n)) + O(1)$. 

Using the results from \cite[Theorem 2.2]{elkthom} again, we know that there exists a non-trivial word $v_k \in \F^{(k)}$ of length bounded from above by roughly $2^{11k/6}$. Applying this with $k= \lfloor 5 \log_9 (2) \cdot \log_2 (\log_2(n)) + C \rfloor$ for some suitable constant $C$, we obtain a word of length bounded by $2^C \cdot \log_2(n)^{ 5 \log_9 (2) \cdot 11/6}$ that is a law for the image $\alpha(G)$ of $G$ in $\prod_p {\rm Aut}(V(p))$. Note that $5 \log_9 (2) \cdot 11/6 \leq 3$.

Burnside showed (see for example \cite[Chapter 5, Theorem 1.4]{MR0231903}) that the kernel of the map from ${\rm Aut}(N(p))$ to ${\rm GL}_{m(p)}(p)$ is a $p$-group and thus, the group $\ker(\alpha \colon G \to {\rm Aut}(N))$ is nilpotent. Clearly, the size of the kernel is also bounded by $n$. Thus, we apply Proposition \ref{nil} to obtain a law of length bounded by $O(\log(n)^{3/2})$ for the kernel. Altogether, we can now use Lemma \ref{ext} to obtain a word of length bounded by $O(\log(n)^{9/2})$ that works for all solvable groups of size at most $n$. 

In other words, we have proved the following proposition:

\begin{proposition} \label{sol}
For every $n \in \Nb$ there exists a word $v_n \in \F_2$ of length bounded by $O(\log(n)^{9/2})$ which is a law for all solvable groups of size at most $n$.
\end{proposition}

Let us remark, that the analysis from above showed that any solvable group of size $n$ has solvability class bounded by $O(\log \log(n))$. This was known from work of Glasby \cite{glasby}. However, Glasby's estimates are not sufficient to establish Proposition \ref{sol} in this form. 
It is easy to see that this double-logarithmic bound is actually sharp. Indeed, the 2-Sylow subgroup of ${\rm Sym}(2^n)$ is the $n$-fold iterated wreath product of $\Zb/2\Zb$. Its cardinality is $2^{2^n-1}$ and its solvability class is equal to $n$.

\section{Semi-simple groups}

Let us first discuss finite non-abelian simple groups before we come to semi-simple groups.
\begin{proposition} \label{simple}
Let $n \in \Nb$. There exists a non-trivial word $w_n \in \F_2$ of length bounded by $O(n/\log(n)^2)$ that is a law for all finite non-abelian simple groups of size at most $n$.
\end{proposition}
\begin{proof}
This uses the CFSG. The worst case in terms of length of laws is the family ${\rm PSL}_2(q)$ with $q=p^k$ some prime power.
It is well-known that the group ${\rm PSL}_2(p^k)$ admits a law of length $O(p^k)$. Indeed, it is easy to see that every element of ${\rm PSL}_2(p^k)$ has order dividing either $p^{k}-1,p$ or $p^k+1$, corresponding to the diagonalizable, the unipotent, and the irreducible case. Thus, combining $a^{p^{k}-1}, a^{p}$ and $a^{p^k+1}$ using Lemma \ref{commut} yields a law for ${\rm PSL}_2(p^k)$ of length $O(p^k)$, where the implied constant does neither depend on $p$ nor $k$.
At the same time the group ${\rm PSL}_2(p^k)$ and has size roughly $p^{3k}$. Combining all such laws for pairs $(p,k)$ with $p^{3k} \leq n$ using Lemma \ref{commut}, we obtain a law with the desired bound on the length. Indeed, it is an easy consequence of the prime number theorem that there are at most $O(n^{1/3}/\log(n))$ prime powers less than $n^{1/3}$. Thus, we obtain a law of length $O( (n^{1/3}/\log(n))^2 n^{1/3})$. This proves the claim for finite simple groups of the form ${\rm PSL}_2(q)$.

For all other families the maximal element order of the group is bounded by $O(n^{1/4})$ by results summarized in work of Kantor-Seress \cite{MR2531224}. Indeed, this is well-known for ${\rm Alt}(n)$, by a classical result of Landau, the maximal order of an element in ${\rm Alt}(n)$ is bounded by $\exp(c (n \log(n))^{1/2})$, for some small and explicit constant. 

The following table summarizes the situation for families of classical Chevalley groups of unbounded rank (twisted or untwisted). The first row contains a lower bound for the size of the group (up to a universal multiplicative factor), the second row contains an upper bound for the maximal element order (again up to some universal multiplicative factor). All information about the maximal element order is taken from the exposition  in \cite{MR2531224}.

$$
\begin{array}{c|cccccccccc}
&{\rm A}_d(q)&{}^2{\rm A}_d(q^2)&{\rm B}_d(q)&{\rm C}_d(q)&{\rm D}_d(q)&{}^2{\rm D}_d(q^2)\\ 
%&d \geq 1&d \geq 2&d \geq 2&d \geq 3&d \geq 4&d \geq 4\\ 
\hline
\mbox{size} &q^{d^2+2d}/d& q^{d^2+2d}/d &q^{2d^2+d}& q^{2d^2+d} &q^{2d^2-d}  &q^{2d^2-d} \\
\mbox{meo} &q^{d}& q^{d} &q^d& q^d & q^d &q^d 
\end{array}
$$

We now come to the families of bounded rank. These contain twisted forms of classical Chevalley groups that only exist in small rank, exceptional Chevalley groups (twisted and untwisted) and Suzuki-Ree groups. Our material stems again from \cite{MR2531224}.

$$
\begin{array}{c|cccccccccc}
&{}^2{\rm B}_2(q)&{}^3{\rm D}_4(q^3)& {\rm F}_4(q)&{}^2 {\rm F}_4(q)& {\rm E}_6(q) &{}^2{\rm E}_6(q^2) & {\rm E}_7(q) & {\rm E}_8(q)  & {\rm G}_2(q) &{}^2{\rm G}_2(q)\\ 
\hline
\mbox{size} &q^5&q^{28}&q^{52} &q^{26}& q^{78} &q^{78}& q^{133} &q^{248}  & q^{14} &q^7\\
\mbox{meo} &q&q^4&q^4 &q^{2}& q^6 &q^6& q^7 & q^8 &q^2 &q
\end{array}
$$
We easily see that the only case for which the maximal element order cannot be bounded by $O(n^{1/4})$ is ${\rm A}_1(q)={\rm PSL}_2(q)$. Since we are only interested in asymptotic bounds, we can ignore sporadic groups and Tits' group.

Thus, we can apply Lemma \ref{commut} to the list $a,a^2,\dots,a^{Cn^{1/4}}$ for some suitable constant $C$ in order to obtain a law of length $O(n^{3/4})$ that applies to all simple groups of size at most $n$ and different from ${\rm PSL}_2(q)$.
Combining now the law for ${\rm PSL}_2(q)$ and the law for the remaining families using Lemma \ref{commut} again, we can finish the proof.
\end{proof}

The study of semi-simple groups goes back to the work of Fitting, see \cite{fitting} for the original reference.
Recall, an isotypical semi-simple group (in the sense of Fitting) is a group $G$ which fits into a chain of inclusions
$$H^k \subset G \subset {\rm Aut}(H) \wr {\rm Sym}(k),$$
where $k$ is a positive integer, the group ${\rm Sym}(k)$ acts on $k$ points, and  $H$ is some finite non-abelian simple group. Let $G$ be an iso-typical semi-simple group of size at most $n$. Denote by $m$ the size of $H$. It is clear that  $m^k \leq n$ and hence $k \leq \log(n)$. Note that the group ${\rm Aut}(H) \wr {\rm Sym}(k)$ fits into an extension
$$1 \to {\rm Aut}(H)^k \to {\rm Aut}(H) \wr {\rm Sym}(k) \to {\rm Sym}(k) \to 1.$$
By a result of Kozma and the author \cite{kozthom} there exists group law for ${\rm Sym}(k)$ of length bounded by $\exp(C \log(k)^4 \log \log(k)).$ This result also relies on the CFSG. For the purpose of this proof, we can work with an explicit law of length bounded by $\exp(C (k \log(k))^{1/2})$, see \cite{brm} or \cite{kozthom} for a discussion. 
Moreover, the automorphism group ${\rm Aut}(H)$ of a non-abelian finite simple group $H$ fits into an exact sequence
$$1 \to H \to {\rm Aut}(H) \to {\rm Out}(H) \to 1,$$ where the group ${\rm Out}(H)$ solvable of class $3$. This assertion is usually called Schreier's Conjecture and was only proved as a consequence of the CFSG. Thus, if $H$ admits a law of length bounded by $O(m/\log(m)^2)$ by Proposition \ref{simple}, then Lemma \ref{ext} allows to construct a law of length bounded by $O(m/\log(m)^2)$ also for ${\rm Aut}(H)$. Note that this will also be a law for ${\rm Aut}(H)^k$. Using Lemma \ref{ext} again, we can combine this with the law for ${\rm Sym}(k)$ and obtain a law for $G$.
The length will be bounded by
$$O(m \log(m)^{-2} \exp(C (k \log(k))^{1/2}),$$
where the construction only depended on $m$ and $k$.
We treat the case $k=1,m=n$ and $k =\log(n),m=n^{1/2}$ separately and combine the results using  Lemma \ref{commut}.
Note that each isotypical semi-simple group of cardinality less or equal $n$ will be covered by one of the two cases, since if $k \geq 2$ implies $m \leq n^{1/2}$. Thus, the resulting word will be a law for all isotypical semi-simple groups of size at most $n$.
It is easy to see that the dominating contribution comes from $k=1,m=n$. This proves Proposition \ref{simple} for isotypical semi-simple groups.

A group is called semi-simple (in the sense of Fitting) if it does not contain any non-trivial abelian normal subgroups, see for example \cite[p. 89]{MR1357169} for a modern reference. (Note that the term {\it semi-simple} is ambiguously used in group theory.) By a result of Fitting, a general semi-simple group $G$ is contained in a product of isotypical semi-simple groups $G_{(H,k)}$, where $H^k \subset G_{(H,k)} \subset {\rm Aut}(H) \wr S_k$ and also $H^k \subset G$. 
Since all bounds depended only on the fact that $|H^k| \leq |G|$, the proof of Proposition \ref{simple} for isotypical semi-simple groups extends verbatim to the case of general semi-simple groups. 

Thus, we can summarize our analysis as follows.
\begin{proposition} \label{semi-simple}
Let $n \in \Nb$. There exists a non-trivial word $w_n \in \F_2$ of length bounded by $O(n/\log(n)^2)$ that is a law for all finite semi-simple groups of size at most $n$.
\end{proposition}

This finishes the discussion of the case of semi-simple groups.

\section{General case}

We are now ready to prove our main result. The general case is done by combining our results for solvable and for semi-simple groups. By a result of Fitting \cite{fitting}, every finite group has a solvable normal subgroup such that the quotient is semi-simple. 

Let $G$ be a finite group of size at most $n$ with solvable normal subgroup $S$ and semi-simple quotient $L:=G/S$. Clearly, either $|S| \leq \log(n)^{9/2}$ or $|L| \leq n/\log(n)^{9/2}$. We apply Lemma \ref{ext} in both situations and combine the laws that we obtain from Proposition \ref{sol} and Proposition \ref{semi-simple}. In the first case, a word of length bounded by $O(\log(\log(n)^{9/2})^{9/2} \cdot n \log(n)^{-2})$ exists that serves as a law. In the second case, we obtain a word of length bounded by $O(\log(n)^{9/2} \cdot n/\log(n)^{9/2}  \log(n/\log(n)^{9/2})^{-2}) = O(n \log(n/\log(n)^{9/2})^{-2})$. 
Combining the two cases using Lemma \ref{commut}, we obtain a word of length $O(n \log\log(n)^{9/2}/\log(n)^2)$ and claim that it is a law for all groups of size at most $n$.
Indeed, any finite group of size at most $n$ will satisfy one of the two conditions -- finally, this proves our main theorem.
\begin{theorem} \label{main}
For every $n \in \Nb$ there exists a non-trivial word $w_n \in \F_2$ of length $$O\left(\frac{n\log\log(n)^{9/2}}{\log(n)^{2}} \right)$$ that is a law for all finite groups of size at most $n$. In particular, the length of $w_n$ is sublinear in $n$.
\end{theorem}

Note that the law that is constructed in the proof of the previous theorem is a completely explicit combination of powers and commutators -- the construction itself does not depend on the CFSG. Note also that the proof depends on the CFSG only in using the following two facts:
\begin{enumerate}
\item[(i)] The order of elements in a non-abelian simple group $G$ different from ${\rm PSL}_2(p^k)$ is bounded from above by $O(|G|^{1/4})$.
\item[(ii)] The outer automorphism group of a non-abelian finite simple group is solvable of class at most $3$.
\end{enumerate}

In Theorem \ref{main}, there is room for improvement in the factor $\log\log(n)^{9/2}$ (even though we were not able to remove it completely) and we believe that $n/\log(n)^2$ is close to the truth, i.e., that the result in Theorem \ref{main} is almost sharp. 
However, it seems notoriously difficult to prove non-trivial lower bounds for the length of laws. This problem also arises in the study of the symmetric group, see \cite{kozthom}, where the best known lower bounds for the length of a law for ${\rm Sym}(n)$ are still linear in $n$. In that direction it would also be very interesting to determine if the bounds in the proof of Proposition \ref{simple} are sharp. This is of interest since ${\rm Sym}(n)$ contains ${\rm PSL}_2(p^k)$ for all $p^k<n$ and hence, a result in this direction would provide a non-trivial lower bound for the length of laws of ${\rm Sym}(n)$. However, any of this seems currently out of reach.

\section*{Acknowledgments}

This research was supported by ERC Starting Grant No.\ 277728. I thank Martin Kassabov for interesting discussions on this topic and helpful comments on this paper. I thank Laci Pyber for pointing out the work of Glasby \cite{glasby}.

\end{document}